\documentclass[amsthm]{elsart}
\usepackage{graphicx,amssymb}
\usepackage{amsmath}
\usepackage{longtable}
\usepackage{float}
\usepackage{ifpdf}
\usepackage{amsthm,amssymb,amsmath}
\usepackage{float}
 \usepackage{graphicx}
\usepackage{pgf,tikz}
\usetikzlibrary{arrows}
\usepackage{epsfig}
\usepackage{epstopdf}
\usepackage{color}
\usepackage{pgfplots}

\usepackage{amssymb,mathrsfs} 
\usepackage{amsfonts}
\usepackage{graphicx}
\usepackage{tikz}
\usepackage{color}%
\ifpdf
\usepackage[%
  pdftitle={Instructions for use of the document class
    elsart},%
  pdfauthor={Simon Pepping},%
  pdfsubject={The preprint document class elsart},%
  pdfkeywords={instructions for use, elsart, document class},%
  pdfstartview=FitH,%
  bookmarks=true,%
  bookmarksopen=true,%
  breaklinks=true,%
  colorlinks=true,%
  linkcolor=blue,anchorcolor=blue,%
  citecolor=blue,filecolor=blue,%
  menucolor=blue,pagecolor=blue,%
  urlcolor=blue]{hyperref}
\else
\usepackage[%
  breaklinks=true,%
  colorlinks=true,%
  linkcolor=blue,anchorcolor=blue,%
  citecolor=blue,filecolor=blue,%
  menucolor=blue,pagecolor=blue,%
  urlcolor=blue]{hyperref}
\fi

\makeatletter
\def\elsartstyle{%
    \def\normalsize{\@setfontsize\normalsize\@xiipt{14.5}}
    \def\small{\@setfontsize\small\@xipt{13.6}}
    \let\footnotesize=\small
    \def\large{\@setfontsize\large\@xivpt{18}}
    \def\Large{\@setfontsize\Large\@xviipt{22}}
    \skip\@mpfootins = 18\p@ \@plus 2\p@
    \normalsize
}
\@ifundefined{square}{}{\let\Box\square}
\makeatother

\begin{document}

\begin{frontmatter}
\title{Changing and unchanging 2-rainbow independent domination}

\author[B1]{Pu Wu}\ead{puwu1997@126.com},
\author[B1]{Zehui Shao\corauthref{cor}}\ead{zshao@gzhu.edu.cn},
\author[B2]{Vladimir Samodivkin}\ead{vl.samodivkin@gmail.com},
\author[B3]{S.M. Sheikholeslami}\ead{s.m.sheikholeslami@azaruniv.ac.ir},
\author[B3]{M. Soroudi}\ead{m.soroudi@azaruniv.ac.ir},
\author[B4]{Shaohui Wang}\ead{shaohuiwang@yahoo.com}
\journal{Discrete Mathematics \& Theoretical Computer Science}

\corauth[cor]{Corresponding author.}

\address[B1]{Institute of Computing Science and Technology Guangzhou University, Guangzhou 510006, China}
\address[B2]{Department of Mathematics, University of Architecture, Civil Engineering and Geodesy, Sofia, Bulgaria}
\address[B3]{Department of Mathematics, Azarbaijan Shahid Madani University Tabriz, I.R. Iran}
\address[B4]{Department of Mathematics, Savannah State University, Savannah, GA 31404, USA}

\begin{abstract}
For a function $f : V(G ) \rightarrow \{0, 1, 2\}$  we denote by $V_i$ the set of vertices to which the value $i$ is assigned by
$f$, i.e. $V_i = \{ x \in V (G ) : f(x ) = i \}$.
If a function $f: V(G) \rightarrow \{0,1,2\}$  satisfying the condition that
$V_i$ is independent for $i \in \{1,2\}$ and
every vertex $u$ for which $f(u) = 0$
is adjacent to at least one vertex $v$ for which $f(v) = i$ for each $i \in \{1,2\}$,
then $f$ is called a 2-rainbow independent dominating function (2RiDF).
The weight $w(f)$ of a 2RiDF $f$ is the value $w(f) = |V_1|+|V_2|$. The minimum weight of
a 2RiDF on a graph $G$ is called the \emph{2-rainbow independent domination
number} of $G$.
A graph $G$ is 2-rainbow independent domination stable if the 2-rainbow independent domination
number of $G$ remains unchanged under removal of any vertex.
In this paper, we characterize 2-rainbow independent domination stable trees and we study the effect of edge removal on 2-rainbow independent domination
number in trees.
\end{abstract}

\end{frontmatter}
\newtheorem{thm1}{Theorem}
\newtheorem{clm1}{Claim}
\newtheorem{fact1}{Fact}
\newtheorem{rem1}{Remark}
\newtheorem{lem1}{Lemma}
\newtheorem{cor1}{Corollary}
\newtheorem{def1}{Definition}
\newtheorem{con1}{Construction}
\newtheorem{ex1}{Example}
\newtheorem{prop1}{Proposition}
\newtheorem{prob1}{Problem}
\newtheorem{op1}{Operation}
\newtheorem{obv}{Obervation}
\newtheorem{prelem}{{\bf Proposition}}
\newtheorem{theorem}{Theorem}
\newtheorem{corollary}[theorem]{Corollary}
\newtheorem{lemma}[theorem]{Lemma}
\newtheorem{observation}[theorem]{Observation}
\newtheorem{proposition}[theorem]{Proposition}
\newtheorem{exa}[theorem]{Example}
\newtheorem{p}{Problem}
\newtheorem{con}{Conjecture}
\theoremstyle{definition}
\newtheorem{definition}[theorem]{Definition}
\theoremstyle{remark}


\newtheorem{them}{\bf Theorem}
\newtheorem{lema}[them]{\bf Lemma}
\renewcommand{\thethem}{\Alph{them}}
\renewcommand{\thelema}{\Alph{lema}}
\newtheorem{observ}[them]{\bf Observation}
\renewcommand{\theobserv}{\Alph{observ}}
\renewcommand{\theprop}{\Alph{\prop}}
\renewcommand{\thecor}{\Alph{cor}}

\section{Introduction}
In this paper, we only consider finite graphs without loops and multiple edges.
For notation and graph theory terminology we follow \cite{hhs1} in general.
Let $G=(V,E)$ be a graph with vertex set $V(G)=V$ and edge set $E(G)=E$. The {\em order} $|V|$ of $G$ is denoted by $n=n(G)$.
For any vertex $v \in V(G)$, the \emph{ open neighborhood} of $v$ is the set  $N(v)=\{u\in V(G)\mid uv\in E(G)\}$, and the {\em closed neighborhood} of $v$ is $N[v]=\{v\} \cup N(v)$.
The \emph{degree} of a vertex $v\in V$ is $\deg(v)=\deg _{G}(v)=|N(v)|$.
The \emph{distance} of two vertices $u$ and $v$ in $G$, denoted by $d_G(u,v)$, is the length of a shortest path between $u$ and $v$. The
\emph{diameter} ${\rm diam}(G)$ of a graph $G$ is the
greatest distance between two vertices of $G$. A \emph{leaf} is a vertex of degree 1,
a \emph{support vertex} is a vertex adjacent to a leaf, and a \emph{strong support
vertex} is a support vertex adjacent to at least two leaves. For a vertex
$v$ in a (rooted) tree $T$, let $C(v)$ and $D(v)$ denote the set of
children and descendants of $v$, respectively and let $D[v]=D(v)\cup \{v\}$.
Also, the {\em depth of $v$}, depth($v$), is the largest distance from $v$
to a vertex in $D(v)$.
The \emph{maximal subtree} at $v$ is the subtree of $T$ induced by $D[v]$,
and is denoted by $T_{v}$.  A graph is \emph{trivial} if it has a single vertex. We write $P_n$ for the path on $n$ vertices and $C_n$ for the cycle on $n$ vertices.
For a graph $G$, let $I(G)=\{v|v\in V(G),\deg(v)=1\}$ and $L(v)=N(v)\cap I(G)$.
A {\em double star} $DS_{p,q}$ is a tree containing exactly two non-pendant vertices which one is adjacent to $p$ leaves and other is adjacent to $q$ leaves.

A set $S\subseteq V$ in a graph $G$ is a \textit{dominating set} if every
vertex of $G$ is either in $S$ or adjacent to a vertex of $S.$ The \textit{%
domination number} $\gamma (G)$ equals the minimum cardinality of a
dominating set in $G$. An {\em efficient
dominating set} in a graph $G$ is a set $S \subseteq V (G)$ such that $\cup_{s \in S}N[s]$ is a partition
of $V(G)$. For a comprehensive treatment of domination parameters in
graphs, see the monographs by Haynes, Hedetniemi, and Slater \cite{hhs1, hhs2,ww}.
An {\em independent dominating set}  of $G$ is a dominating set of $G$ which is independent.
The minimum cardinality of
an independent dominating set on a graph $G$ is called the \emph{independent domination
number} $i(G)$ of $G$.
An independent set $S$ of a graph $G$ is called an $i$-set if $|S|=i(G)$.
This graph-theoretical invariant has been explored extensively in the literature; for an
illuminating survey the reader is referred to Goddard and Henning \cite{hhh}.
Independent dominating
sets and their combination with other domination parameters have been studied extensively in the literature, e.g. with rainbow domination \cite{Shao}, Roman domination \cite{c1,c2,sam}; see for example the books \cite{hhs1, hhs2}.

For a function $f : V(G ) \rightarrow \{0, 1, 2\}$  we denote by $V_i$ the set of vertices to which the value $i$ is assigned by
$f$, i.e. $V_i = \{ x \in V (G ) : f(x ) = i \}$.
If a function $f \rightarrow \{0,1,2\}$  satisfying the condition that
$V_i$ is independent for $i \in \{1,2\}$ and
every vertex $u$ for which $f(u) = 0$
is adjacent to at least one vertex in $V_i$ for each $i \in \{1,2\}$,
then $f$ is called a 2-rainbow independent dominating function (2RiDF).
The weight $w(f)$ of a 2RiDF $f$ is the value $w(f) = \sum_{i=1}^2|V_i|$. The minimum weight of
a 2RiDF on a graph $G$ is called the \emph{2-rainbow independent domination
number} of $G$. The concept of $k$-rainbow independent  dominating function was
first defined by \v{S}umenjak et al. \cite{Sumenjak}.
A graph $G$ is {\em 2-rainbow independent domination stable} ($\gamma_{ri2}$-stable) if the 2-rainbow independent domination number of $G$ remains unchanged
under removal of any vertex. A graph $G$ is called a $\gamma_{ri2}$-ER-critical graph if the 2-rainbow independent domination number of $G$ changed under removal of any edge.

In this paper, we first study basic properties of 2-rainbow independent domination stable. Then we characterize 2-rainbow independent domination stable trees and we investigate the effect of edge removal on 2-rainbow independent domination number in trees.

We make use of the following results in this paper.


\begin{prop}\label{leaf}\cite{Sumenjak}
Let $x$ be a leaf of a nontrivial tree $T$.
 Then $i(T)-1 \leq i(T-x) \leq i(T)$ and $\gamma_{ri2}(T)-1 \leq \gamma_{ri2}(T-x)
 \leq\gamma_{ri2}(T)$.
\end{prop}
\begin{corollary}\label{+}
There is no tree $T$ such that $\gamma_{ri2}(T-x) > \gamma_{ri2}(T)$
for each vertex $x$ of $T$.
\end{corollary}

\begin{observ} \label{obv}\cite{Sumenjak}
If $G$ is a graph without isolated vertices and $(V_0; V_1; V_2)$
is a $2RiDF$ of $G$, then every leaf of $G$ belongs to $V_1\cup V_2$.
\end{observ}
\begin{corollary}\label{leaf2}
 Let $x$ be a leaf of a graph $G$ and $y$ its neighbor.
If $f(y) \not = 0$ for some  $\gamma_{ri2}$-function $f$ on $G$,
then $\gamma_{ri2}(G-x) < \gamma_{ri2}(G)$.
\end{corollary}

\begin{prop}\label{path}\cite{Sumenjak}
$\gamma_{ri2}(P_n) = \left\lceil (n+1)/2 \right\rceil$,
$\gamma_{ri2}(C_m) = \left\lceil m/2 \right\rceil$
if $m\equiv 0,3\pmod 4$  and
$\gamma_{ri2}(C_m) = \left\lceil m/2 + 1 \right\rceil$
if $m\equiv 1,2\pmod 4$.
\end{prop}
\begin{def1} \label{exsp}
Let $P_3^i:=  v_1^iv_2^iv_3^i\;(i=1,2,\ldots,k)$, be $k \geq 2$ vertex disjoint paths and
let $S_k^{v_1}$ be the tree obtained from the union of all $P^i_3$ by adding a vertex $v_1$ and joining it to $v_1^1, v_1^2,\ldots,v_1^k$.
Clearly, $S_k^{v_1}$ is a spider with head $v_1$ and feet $P^i_3\; (1\le i\le k)$. Clearly, $\gamma_{ri2}(S_k^{v_1}) = 2k$.
\end{def1}

We will use the following definitions.
\begin{def1}
For a graph $G$, define  \\
$W_{i}(G) = \{ v\in V(G)\mid \textrm{$f(v)=i$  for any $\gamma_{ri2}(G)$-function $f$}\}$.
\end{def1}

\begin{observation}\label{esp}
If $k \geq 3$, then $S_k^{v_1}$ is a  $\gamma_{ri2}$-stable graph, $W_0(S_k^{v_1}) = \{v_1\} \cup \{v_2^1,v_2^2,\ldots,v_2^k\}$,
and for each $\gamma_{ri2}$-function $f$ on    $S_k^{v_1}$, $f(u) \not= 0$ when $u \in V(S_k^{v_1}) -  W_0(S_k^{v_1})$.
\end{observation}


\section{Preliminary results}\label{kr}
In this section we study of basic properties of of 2-rainbow independent domination in graphs.
\begin{proposition}\label{diam}
If $T$ is a tree of order $n$ with $diam(T) =d$, then
$\gamma_{ri2}(T)  \leq n - d + \left\lceil d/2 \right\rceil$.
\end{proposition}
\begin{proof}
Let $P_{d+1}=v_1v_2\ldots,v_{d+1}$ be a diametrical path in $T$ and $f$ a $\gamma_{ri2}$-function on $P_{d+1}$.
For $2\le i\le d$, assume $T_i$ is the component of $T-\{v_{i-1}v_i, v_{i+1}v_i\}$ containing $v_i$ and root $T_i$ at $v_i$. For $i\in \{2,\ldots,d\}$, define $g_i:V(T_i-\{v_i\})\to \{0,1,2\}$ as follows: if $f(v_i)=0$, then $g_i(x)=1$ when $d(v_i,x)$ is odd and $g_i(x)=2$ when $d(v_i,x)$ is even, and if $f(v_i)\neq 0$, say $f(v_i)=1$, then $g_i(x)=2$ when $d(v_i,x)$ is odd and $g_i(x)=1$ when $d(v_i,x)$ is even. Define $g:V(T)\to \{0,1,2\}$ by $g(x)=f(x)$ for $x\in V(P_{d+1})$ and $g(x)=g_i(x)$ for $x\in V(T_i)-\{v_i\}$ and each $i\in \{2,\ldots,d\}$. Clearly, $g$ is a 2RiDF of $T$ such that
the restriction of $g$ on $T[P_{d+1}]$ is $f$ and $V_0^g=V_0^f$.
Hence $\gamma_{ri2}(T) \leq \omega(g) = n-d-1+\gamma_{ri2}(P_{d+1})$ and the result follows by Proposition \ref{path}.
\end{proof}
Next result is an immediate consequence of Proposition \ref{diam}.

\begin{corollary}\label{n-1}
Let $T$ be a tree of order $n\ge 4$. Then $\gamma_{ri2}(T) = n-1$
if and only if $T$ is a star or a double star $DS_{1,n-3}$.
\end{corollary}
The proof of next result is similar to the proof of Proposition \ref{diam} and therefore omitted.
\begin{proposition}\label{diam1}
If $G$ is a unicyclic  graph of order $n$ with girth $m$, then $\gamma_{ri2}(G) = n-\left\lfloor m/2 \right\rfloor$
when $m\equiv 0,3\pmod 4$  and
$\gamma_{ri2}(G) = n+1-\left\lfloor m/2 + 1 \right\rfloor$
when $m\equiv 1,2\pmod 4$.
\end{proposition}
\begin{proposition}\label{i-1}
If $S$ is a clique in a graph $G$, then $\gamma_{ri2}(G-S)\geq \gamma_{ri2}(G)-2$.
\end{proposition}
\begin{proof}
Let $f$ be a $\gamma_{ri2}(G-S)$-function. If $\{1,2\}\subseteq \{f(u)\mid u\in N(x)\}$ for each $x\in S$, then $f$ can be extended to a 2RiDF of $G$ by assigning a 0 to all vertices in $S$ and so $\gamma_{ri2}(G)\le \gamma_{ri2}(G-S)$. Let $\{1,2\}\not\subseteq \{f(u)\mid u\in N(x)\}$ for some $x\in S$. Assume without loss of generality that $1\not\in \{f(u)\mid u\in N(x)\}$. Define $g:V(G-(S-\{x\}))\to \{0,1,2\}$ by $g(x)=1$ and $g(u)=f(u)$ otherwise. Then $1\in \{f(u)\mid u\in N(y)\}$ for each $y\in S-\{x\}$. If $2\in \{f(u)\mid u\in N(y)\}$ for each $y\in S-\{y\}$, then $g$ can be extended to a 2RiDF of $G$ by assigning a 0 to all vertices in $S-\{x\}$ and so $\gamma_{ri2}(G)\le \gamma_{ri2}(G-S)+1$. Suppose $2\not\in \{f(u)\mid u\in N(y)\}$ for some $y\in S-\{x\}$. Then $g$ can be extended to a 2RiDF of $G$ by assigning a 2 to $y$ and a 0 to all vertices in $S-\{x,y\}$ implying that $\gamma_{ri2}(G)\le \gamma_{ri2}(G-S)+2$. Thus $\gamma_{ri2}(G-S)\geq \gamma_{ri2}(G)-2$.
\end{proof}

\begin{proposition}\label{-1}
Let $x$ be a vertex of a graph $G$.
Then $$\gamma_{ri2}(G) -1 \leq \gamma_{ri2}(G-x) \leq \gamma_{ri2}(G) + \deg(x) -1.$$
\end{proposition}
\begin{proof}
A close look at the proof of Proposition \ref{i-1}, leads to the left inequality.

To prove the right side inequality, let $(V_0;V_1;V_2)$ be a $\gamma_{ri2}$-function of $G$.
If $x \in V_0$, then $(V_0-\{x\};V_1;V_2)$ is a 2RiDF of $G-x$. Now
let $x \not\in V_0$ and without loss of generality suppose $x \in V_1$.
If $(V_0-\{x\};V_1;V_2)$  is a 2RiDF of $G-x$, then we are done.
Denote by $A$  the set of all neighbors of $x$ in $V_0$
such that each of them has no neighbor in $V_1- \{x\}$.
Let $B$ be an $i$-set of $G[A]$.
Then $(V_0-(B\cup\{x\});V_1 \cup B;V_2)$ is a 2RiDF of $G-x$.
It remains to note that $|B| \leq |A| \leq \deg(x)-1$.
\end{proof}
A {\em subdivision} of an edge $uv$ is obtained by replacing the
edge $uv$ with a path $uwv$, where $w$ is a new vertex.
The {\em subdivision graph} $S(G)$ is the graph obtained
from $G$ by subdividing each edge of $G$.
\begin{rem}\label{-1+}
If $T$ is a tree obtained from  $K_{1,n}\;(n \geq 3)$ centered at $x$, then
$\gamma_{ri2}(S(K_{1,n})-x) = \gamma_{ri2}(S(K_{1,n})) + \deg(x) -1$.
\end{rem}

\begin{corollary}\label{-=+}
Let $x$ be a vertex of a graph $G$.
If $f(x)=0$ for some $\gamma_{ri2}$-function $f$ on $G$, then $\gamma_{ri2}(G-x)\leq \gamma_{ri2}(G)$.
If $\gamma_{ri2}(G) < \gamma_{ri2}(G-x)$, then
$l(x) \not= 0$ for each $\gamma_{ri2}$-function $l$ on $G$.
\end{corollary}
\begin{theorem}\label{-1-1}
Let $T$ be a tree. Then for each vertex $x\in V(T)$, $\gamma_{ri2}(T-x)= \gamma_{ri2}(T)-1$  if and only if $T \in \{K_1,K_2\}$.
\end{theorem}
\begin{proof}
The sufficiency is clear.  Let for each vertex $x\in V(T)$ we have $\gamma_{ri2}(T-x)= \gamma_{ri2}(T)-1$. Suppose, to the contrary, that $|V(T)|\ge 3$. Let $v_1v_2\ldots v_k\;(k\ge 3)$ be a diametrical path in $T$. First let $\deg(v_2)\ge 3$ and $w\in L(v_2)-\{v_1\}$. Suppose $f$ is a $\gamma_{ri2}(T-v_2)$-function. We may assume without loss of generality that $f(v_1)=1$ and $f(w)=2$. Then $f$ can be extended to a 2RiDF of $T$ by assigning a 0 to $v_2$ which leads to a contradiction by $\gamma_{ri2}(T-v_2)= \gamma_{ri2}(T)-1$. Assume now that $\deg(v_2)=2$. Let $f$ be a $\gamma_{ri2}(T-v_3)$-function. We may assume without loss of generality that $f(v_1)=1$ and $f(v_2)=2$. If $1\in \{f(u)\mid u\in N(v_3)\}$, then $f$ can be extended to a 2RiDF of $T$ by assigning a 0 to $v_3$ which leads to a contradiction again. Suppose $1\not\in \{f(u)\mid u\in N(v_3)\}$. Then the function $h:V(T)\to \{0,1,2\}$ defined by $h(v_3)=1, h(v_2)=0, h(v_1)=2$ and $h(x)=f(x)$ otherwise, is a 2RIDF of $T$ of weight $\gamma_{ri2}(T)-1$ which is a contradiction. Thus $T \in \{K_1,K_2\}$ and the proof is complete.
\end{proof}

\begin{proposition}\label{pen3}
Let $P_3^i: v_1^iv_2^iv_3^i\;(i = 1,2,\ldots,k)$ be  $k$ vertex disjoint  paths in a graph $G$,
where  $v_3^i$ is a leaf  and $\deg(v_2^i) = \deg(v_1^i) = 2$ for each $i = 1,2,\ldots,k$.
\begin{itemize}
\item[(i)]  If $g(v_2^j) \not = 0$ for some $\gamma_{ri2}(G)$-function $g$ and some $j \in \{1,2,\ldots,k\}$,  then
there is a $\gamma_{ri2}(G)$-function $\ell$  such that
$\ell|_{G-\{v_1^j,v_2^j,v_3^j\} }= g|_{G-\{v_1^j,v_2^j,v_3^j\} },  \ell(v_2^j)=0$, and $\{\ell(v_1^j), \ell(v_3^j)\} = \{1,2\}$.
\item[(ii)]  There is a $\gamma_{ri2}$-function $f$ on $G$ with $f(v_2^i) = 0$ for $i\in\{1,2,\ldots,k\}$. 									
\end{itemize}
\end{proposition}
\begin{proof}
(i)  Assume without loss of generality that $g(v_2^1)=2$ for some $\gamma_{ri2}(G)$-function $g$.
By Observation \ref{obv}, we have $f(v_3^1)=1$. Since $g$ is a $\gamma_{ri2}(G)$-function, $g(v_1^1) = 0$.
Let $N(v_1^1 )= \{v_2^1, u\}$. To dominate $v_1^1$, we must have $g(u)=1$. Then the function $\ell:V(G)\to \{0,1,2\}$ defined by $\ell(v_1^1)=2, \ell(v_2^1)=0$ and $\ell(x)=g(x)$ otherwise, is a $\gamma_{ri2}(G)$-function satisfying the condition.
	
(ii) Suppose $f$ is $\gamma_{ri2}$-function on $G$ such that $s=|\{v_2^j\mid g(v_2^j)\neq 0\}|$ is as small as possible.
By (i) we obtain $s=0$ and we are done.
\end{proof}

\begin{proposition}\label{k12}
Let $G$ be a graph and $y\in V(G)$. If $G_y$ is a graph obtained from $G$
by adding a path $v_2v_1v_3$ and possibly joining $v_1$ to $y$, then  $\gamma_{ri2}(G_y) = \gamma_{ri2}(G) + 2$.
\end{proposition}
\begin{proof}
If $v_1$ is not joined to $y$, then clearly $\gamma_{ri2}(G_y)=\gamma_{ri2}(G) +2$. Suppose $v_1$ is joined to $y$. Clearly, any $\gamma_{ri2}(G)$-function can be extended to a {\rm 2RiDF} of $G_y$ by assigning a 1 to $v_2$, a 2 to $v_3$ and a 0 to $v_1$ and this implies that $\gamma_{ri2}(G_y) \leq \gamma_{ri2}(G) +2$.

Now, let $h$ be a $\gamma_{ri2}$-function on $G_y$. By Observation \ref{obv}, $v_2,v_3\in V_1\cup V_2$. If $h(v_1)=0$ or $h(y)\neq 0$ or ($h(v_1)\neq 0, h(y)=0$ and $h(v_1)\in \{h(u)\mid u\in N_G(y)\}$), then the function $h$ restricted to $G$ is a $2RiDF$ of $G$ and so $\gamma_{ri2}(G_y) \geq \gamma_{ri2}(G) +2$. Hence we assume $h(v_1)\neq 0, h(y)=0$ and $h(v_1)\not\in \{h(u)\mid u\in N_G(y)\}$. Then the function $g:V(G)\to \{0,1,2\}$ defined by $g(y)=h(v_1)$ and $g(u)=h(u)$ otherwise, is a {\rm 2RiDF} of $G$ and so $\gamma_{ri2}(G_y) \geq \gamma_{ri2}(G) +2$. Thus $\gamma_{ri2}(G_y)=\gamma_{ri2}(G) +2$.
\end{proof}

\begin{proposition}\label{k13}
Let $G$ be a graph and $y\in V(G)$. If $G_y$ is a graph obtained from $G$
by adding a path $v_5v_4v_3 v_2v_1$ and possibly joining $v_4$ to $y$, then $\gamma_{ri2}(G_y) = \gamma_{ri2}(G) + 3$.
\end{proposition}
\begin{proof}
If $v_4$ is not joined to $y$, then clearly $\gamma_{ri2}(G_y)=\gamma_{ri2}(G) +3$. Suppose $v_4$ is joined to $y$. Clearly, any $\gamma_{ri2}(G)$-function can be extended to a $2RiDF$ of $G_y$ by assigning a 1 to $v_1,v_5$, a 2 to $v_3$ and a 0 to $v_2,v_4$, and this implies that $\gamma_{ri2}(G_y) \leq \gamma_{ri2}(G) +3$.

Now, let $h$ be a $\gamma_{ri2}$-function on $G_y$. By Proposition \ref{pen3}, we may assume that $h(v_1)=1, h(v_2)=0$ and $h(v_3)=2$. Also by Observation \ref{obv}, $h(v_5)\neq 0$. If $h(v_4)=0$ or $h(y)\neq 0$ or ($h(v_4)\neq 0, h(y)=0$ and $h(v_4)\in \{h(u)\mid u\in N_G(y)\}$), then the function $h$ restricted to $G$ is a $2RiDF$ of $G$ and so $\gamma_{ri2}(G_y) \geq \gamma_{ri2}(G) +3$. Hence we assume $h(v_4)\neq 0, h(y)=0$ and $h(v_4)\not\in \{h(u)\mid u\in N_G(y)\}$. Then the function $g:V(G)\to \{0,1,2\}$ defined by $g(y)=h(v_4)$ and $g(u)=h(u)$ otherwise, is a 2RiDF of $G$ and so $\gamma_{ri2}(G_y) \geq \gamma_{ri2}(G) +3$. Thus $\gamma_{ri2}(G_y)=\gamma_{ri2}(G) +3$.
\end{proof}
\begin{proposition}\label{spider}
Let $G$ be a graph and $y\in V(G)$. If $G_y$ is a graph obtained from $G$ by adding $S_k^{v_1}\; (k\ge 2)$  and possibly joining  $v_1$ to $y$, then   $\gamma_{ri2}(G_y) = \gamma_{ri2}(G) + 2k$.
\end{proposition}
\begin{proof}
We use the notation  of Definition \ref{exsp}. If $v_1$ is not joined to $y$, then clearly $\gamma_{ri2}(G_y)=\gamma_{ri2}(G) +2k$. Clearly, any $\gamma_{ri2}(G)$-function can be extended to a 2RiDF of $G_y$ by assigning a 1 to $v_3^1,\ldots,v_3^k$, a 2 to $v_1^1,\ldots,v_1^k$ and a 0 to $v_1, v_2^1,\ldots,v_2^k$, and this yields $\gamma_{ri2}(G_y) \leq \gamma_{ri2}(G) +2k$.

Now, let $h$ be a $\gamma_{ri2}$-function on $G_y$. By Pproposition \ref{pen3}, we may assume that $h$ assigns a 1 to $v_3^1,\ldots,v_3^k$, a 2 to $v_1^1,\ldots,v_1^k$ and a 0 to $v_2^1,\ldots,v_2^k$. Using an argument similar to that described above Proposition, we obtain $\gamma_{ri2}(G_y) \geq \gamma_{ri2}(G) +2k$. Hence $\gamma_{ri2}(G_y) = \gamma_{ri2}(G) + 2k$.
\end{proof}
\section{2-rainbow independent domination stable graphs}
I this section we classify all 2-rainbow independent domination stable trees. First we present some classes of graphs which are not 2-rainbow independent domination stable.
\begin{proposition}\label{k14}
Let $G$ be a graph and $x\in V(G)$.
\begin{enumerate}
  \item If $G_x$ is a graph obtained from $G$ by adding three new vertices $v_1,v_2,v$ and joining $v$ to $x$ and $v_2$ to $x,v_1$, then $G_x$ is not a $\gamma_{ri2}$-stable graph.
  \item If $G_x$ is a graph obtained from $G$ by adding two paths  $P_2:=v_1v_2$ and $P_2:=v_1'v_2'$ and joining $x$ to $v_2,v_2'$, then $G_x$ is not a $\gamma_{ri2}$-stable graph.
  \item If $G_x$ is a graph obtained from $G$ by adding two paths $P_2:= v_1v_2$  and  $P_3:= u_1u_2u_3$, and joining $x$ to
$v_2$ and $u_3$,  then $G_x$ is not a $\gamma_{ri2}$-stable graph.
\item If $G_x$ is a graph obtained from $G$ by adding a path $P_4:= u_1u_2u_3u_4$ and joining $x$ to
$u_4$,  then $G_x$ is not a $\gamma_{ri2}$-stable graph.
\item If $G_x$ is a graph obtained from $G$ by adding two paths $P_3:= v_1v_2v_3$  and  $Q_3:= u_1u_2u_3$ and a new vertex $w$, and joining $x$ to $w,v_3$ and $u_3$,  then $G_x$ is not a $\gamma_{ri2}$-stable graph.
\item If $G_x$ is a graph obtained from $G$ by adding $k\ge 3$ pendant edges $xv_1,\ldots,xv_k$, then $G_x$ is not a $\gamma_{ri2}$-stable graph.
\item If $G_x$ is a graph obtained from $G$ by adding a path $P_5:=u_1u_2u_3u_4u_5$ and joining $x$ to $u_4$, then $G_x$ is not a $\gamma_{ri2}$-stable graph.
\end{enumerate}
\end{proposition}
\begin{proof}
\begin{enumerate}
  \item Let $f$ be a $\gamma_{ri2}(G_x)$-function. By Observation  \ref{obv}, $f(v_1)\neq 0$ and $f(v)\neq 0$. To dominate $v_2$, we must have $f(v_2)\neq 0$ or $f(x)\neq 0$. It follows from Corollary \ref{leaf2} that $\gamma_{ri2}(G_x-v_1)<\gamma_{ri2}(G_x)$ or $\gamma_{ri2}(G_x-v)<\gamma_{ri2}(G_x)$. Hence $G_x$ is not a $\gamma_{ri2}$-stable graph.
  \item Suppose, to the contrary, that $G_x$ is a $\gamma_{ri2}$-stable graph. Then we must have $\gamma_{ri2}(G_x) = \gamma_{ri2}(G_x -x)$. Let $f$ be any $\gamma_{ri2}$-function on $G-x$. By Observation \ref{obv}, we have $0\not\in \{f(v_1),f(v_2),f(v_1'),f(v_2')\}$. We may assume that $f(v_1)=f(v_2')=1$ and $f(v_2)=f(v_1')=2$. Then the function $f$ restricted to $G_x-v_1$ is a 2RiDF on $G_x-v_1$ of weight less than $\gamma_{ri2}(G_x)$ which is a contradiction. Thus $G_x$ is not a $\gamma_{ri2}$-stable graph.
  \item Let $f$ be a $\gamma_{ri2}(G_x)$-function. By Observation  \ref{obv}, $f(u_1)\neq 0$ and $f(v_1)\neq 0$. If $f(u_2)\neq 0$ or $f(v_2)\neq 0$, then by Corollary \ref{leaf2}, $G_x$ is not a $\gamma_{ri2}$-stable graph. Let $f(v_2)=f(u_2)=0$. To dominate $v_2$, we must have $f(x)\neq 0$, say $f(x)=1$. Then the function $h:V(G_x-u_1)\to \{0,1,2\}$ defined by $h(u_2)=2$, $h(u_3)=0$ and $h(u)=f(u)$ otherwise, is a 2RiDF on $G_x-v_1$ of weight less than $\gamma_{ri2}(G_x)$ which is a contradiction. Thus $G_x$ is not a $\gamma_{ri2}$-stable graph.
      \item Suppose, to the contrary, that $G_x$ is a $\gamma_{ri2}$-stable graph. Then we must have $\gamma_{ri2}(G_x) = \gamma_{ri2}(G_x -u_3)$. Let $f$ be any $\gamma_{ri2}$-function on $G_x-u_3$. By Observation \ref{obv}, we have $0\not\in \{f(u_1),f(u_2),f(u_4)\}$. We may assume that $f(u_4)=f(u_1)=1$ and $f(u_2)=2$. Then the function $g:V(G_x-u_1)\to \{0,1,2\}$ defined by $g(u_3)=0$ and $g(u)=f(u)$ otherwise, is a 2RiDF on $G_x-u_1$ of weight less than $\gamma_{ri2}(G_x)$ which is a contradiction. Thus $G_x$ is not a $\gamma_{ri2}$-stable graph.
      \item Suppose, to the contrary, that $G_x$ is a $\gamma_{ri2}$-stable graph. Then we must have $\gamma_{ri2}(G_x) = \gamma_{ri2}(G_x -x)$. Let $f$ be any $\gamma_{ri2}$-function on $G_x-x$. We may assume without loss of generality that $f(w)=f(v_1)=f(u_3)=1$, $f(u_1)=f(v_3)=2$ and $f(u_2)=f(v_2)=0$. Then the function $g:V(G_x-w)\to \{0,1,2\}$ defined by $g(x)=0$ and $g(u)=f(u)$ otherwise, is a 2RiDF on $G_x-w$ of weight less than $\gamma_{ri2}(G_x)$ which is a contradiction. Thus $G_x$ is not a $\gamma_{ri2}$-stable graph.
      \item Suppose, to the contrary, that $G_x$ is a $\gamma_{ri2}$-stable graph. Then we must have $\gamma_{ri2}(G_x) = \gamma_{ri2}(G_x -x)$. Let $f$ be any $\gamma_{ri2}$-function on $G_x-x$. We may assume without loss of generality that $f(v_1)=1$ and $f(v_2)=\cdots=f(v_k)=2$. Then the function $g:V(G_x-v_k)\to \{0,1,2\}$ defined by $g(x)=0$ and $g(u)=f(u)$ otherwise, is a 2RiDF on $G_x-v_k$ of weight less than $\gamma_{ri2}(G_x)$ which is a contradiction. Thus $G_x$ is not a $\gamma_{ri2}$-stable graph.
      \item Suppose, to the contrary, that $G_x$ is a $\gamma_{ri2}$-stable graph. Then we must have $\gamma_{ri2}(G_x) = \gamma_{ri2}(G_x -u_3)$. Let $f$ be any $\gamma_{ri2}$-function on $G_x-u_3$. Then $0\not\in \{f(u_5),f(u_1),f(u_2)\}$. We may assume without loss of generality that $f(u_5)=1$. If $f(u_4)\neq 0$, then the function $g:V(G_x-u_1)\to \{0,1,2\}$ defined by $g(u_3)=0, g(u_2)=1$ and $g(u)=f(u)$ otherwise, is a 2RiDF on $G_x-u_1$ of weight less than $\gamma_{ri2}(G_x)-1$ which is a contradiction. If $f(u_4)=0$, then we must have $f(x)=2$ and the function $g:V(G_x-u_5)\to \{0,1,2\}$ defined by $g(u_3)=1, g(u_1)=2, g(u_2)=g(u_4)=0$ and $g(u)=f(u)$ otherwise, is a 2RiDF on $G_x-u_5$ of weight less than $\gamma_{ri2}(G_x)$ which is a contradiction.
          Thus $G_x$ is not a $\gamma_{ri2}$-stable graph.
\end{enumerate}
\end{proof}

\subsection{Trees}
In this subsection we give a constructive characterization of all 2-rainbow independent domination stable trees.

In order to presenting our constructive characterization, we define a family of trees as follows.
Let $\mathcal{T}$ be the family of trees $T$ that can be obtained from a sequence $T_1,T_2,\ldots,T_r$
of trees for some $r \geq 1$, where $T_1\in \{P_3, S^{v_1}_k\mid k \geq 3\}$,  and $T=T_r$.
If $r \geq 2$, $T_{i+1}$ can be obtained from $T_i$ by one of the following three operations.

\begin{description}
\item [Operation ${\mathcal O}_1$: ] If $x \in W_0(T_i)$, then ${\mathcal O}_1$ adds a path $v_2v_1v_3$  and joins $x$ to $v_1$ to
obtain $T_{i+1}$ (see Fig. 1(a) ).
\item [Operation ${\mathcal O}_2$: ] If $x \notin W_0(T_i)$, then ${\mathcal O}_2$ adds a path  $v_4v_3v_2v_1v_5v_6v_7$ and joins $x$ to $v_1$ to obtain $T_{i+1}$ (see Fig. 1(b) ).
\item [Operation ${\mathcal O}_3$: ] If $x \in V(T_i)$, then ${\mathcal O}_3$ adds
  an spider $S_k^{v_1}\;(k \geq 3)$ and joins $x$ to $v_1$ to obtain $T_{i+1}$ (see Fig. 1(c) ).
\end{description}

\begin{figure}[!ht]
\begin{tikzpicture}[scale=0.95,line cap=round,line join=round,>=triangle 45,x=1.0cm,y=1.0cm]
\clip(-0.5,-0.02) rectangle (15.14,4.74);
\draw (0.34,3)-- (1.68,3);
\draw (1.68,3)-- (2.14,3.46);
\draw (1.68,3)-- (2.24,2.74);
\draw(0.66,2.98) circle (0.7cm);
\draw (3.84,3)-- (8,3);
\draw (5.14,3)-- (6.02,3.7);
\draw (6.02,3.7)-- (8,3.7);
\draw(4.1,3) circle (0.7cm);
\draw (9.56,3)-- (14,3);
\draw (11,3)-- (12,3.7);
\draw (12,3.7)-- (14,3.7);
\draw (11,3)-- (12,1.8);
\draw (12,1.8)-- (14,1.8);
\draw(10,3) circle (0.7cm);
\fill [color=black] (0.66,2.98) circle (1.5pt);
\draw[color=black] (0.66,3.2) node {{\tiny $x\in W_0(T_i)$}};
\fill [color=black] (1.68,3) circle (1.5pt);
\draw[color=black] (1.66,2.68) node {$v_1$};
\fill [color=black] (2.14,3.46) circle (1.5pt);
\draw[color=black] (2.18,3.82) node {$v_2$};
\fill [color=black] (2.24,2.74) circle (1.5pt);
\draw[color=black] (2.3,2.42) node {$v_3$};
\draw[color=black] (1.48,1.5) node {$(a)$};
\fill [color=black] (4.1,3) circle (1.5pt);
\draw[color=black] (4.12,3.2) node {\tiny {$x\not\in W_0(T_i)$}};
\fill [color=black] (5.14,3) circle (1.5pt);
\draw[color=black] (5.14,2.7) node {$v_1$};
\fill [color=black] (6,3) circle (1.5pt);
\draw[color=black] (6.04,2.7) node {$v_5$};
\fill [color=black] (7,3) circle (1.5pt);
\draw[color=black] (7,2.72) node {$v_6$};
\fill [color=black] (8,3) circle (1.5pt);
\draw[color=black] (8,2.7) node {$v_7$};
\fill [color=black] (6.02,3.7) circle (1.5pt);
\draw[color=black] (6,4.06) node {$v_2$};
\fill [color=black] (7,3.7) circle (1.5pt);
\draw[color=black] (7,4.06) node {$v_3$};
\fill [color=black] (8,3.7) circle (1.5pt);
\draw[color=black] (7.98,4.08) node {$v_4$};
\draw[color=black] (5.96,1.5) node {$(b)$};
\fill [color=black] (10,3) circle (1.5pt);
\draw[color=black] (9.98,3.32) node {$x$};
\fill [color=black] (11,3) circle (1.5pt);
\draw[color=black] (10.96,3.34) node {$v_1$};
\fill [color=black] (12,3) circle (1.5pt);
\draw[color=black] (11.98,3.3) node {$v_1^2$};
\fill [color=black] (13,3) circle (1.5pt);
\draw[color=black] (12.96,3.36) node {$v_2^2$};
\fill [color=black] (14,3) circle (1.5pt);
\draw[color=black] (14,3.34) node {$v_3^2$};
\fill [color=black] (12,3.7) circle (1.5pt);
\draw[color=black] (11.98,4.08) node {$v_1^1$};
\fill [color=black] (13,3.7) circle (1.5pt);
\draw[color=black] (12.98,4.08) node {$v_2^1$};
\fill [color=black] (14,3.7) circle (1.5pt);
\draw[color=black] (14.08,4.08) node {$v_3^1$};
\fill [color=black] (12,2.62) circle (0.5pt);
\fill [color=black] (12,2.38) circle (0.5pt);
\fill [color=black] (12,2.12) circle (0.5pt);
\fill [color=black] (12,1.8) circle (1.5pt);
\draw[color=black] (12.08,1.46) node {$v_1^k$};
\fill [color=black] (13,1.8) circle (1.5pt);
\draw[color=black] (13.1,1.48) node {$v_2^k$};
\fill [color=black] (14,1.8) circle (1.5pt);
\draw[color=black] (14.12,1.46) node {$v_3^k$};
\draw[color=black] (10.5,1.5) node {$(c)$};
\end{tikzpicture}
\caption{$(a),(b),(c)$ are Operations $\mathcal{O}_1$, $\mathcal{O}_2$ and $\mathcal{O}_3$, respectively.}
\end{figure}
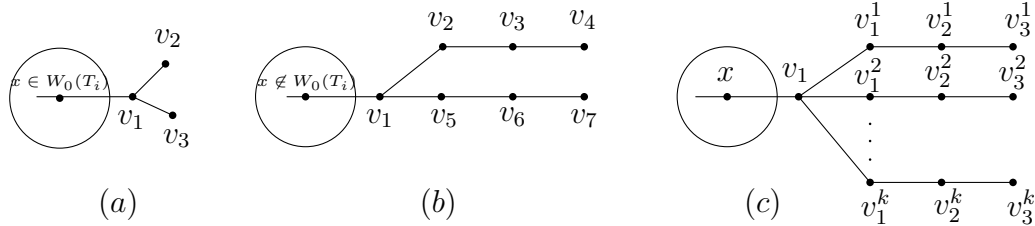

Our main result in this section is the following:
\begin{theorem}\label{main}
Let T be a tree of order $n\geq 3$. Then $T$ is a $\gamma_{ri2}$-stable tree if
and only if $T\in \mathcal{T}$.
\end{theorem}
We proceed with some lemmas.

\begin{lemma}\label{l1}
{\em If $T_i$ is  a $\gamma_{ri2}$-stable tree and a tree $T_{i+1}$ is  obtained from
$T_i$ by Operation ${\mathcal O}_1$, then
$T_{i+1}$ is a $\gamma_{ri2}$-stable tree.}
\end{lemma}
\begin{proof}
By Proposition \ref{k12}, $\gamma_{ri2}(T_{i+1}) = \gamma_{ri2}(T_{i}) +2$.
For any vertex $u$ of $T_{i+1}$, we must show that $\gamma_{ri2}(T_{i+1}-u) = \gamma_{ri2}(T_{i+1})$.
If $u \in V(T_i)-\{x\}$, then by Proposition \ref{k12} we have
$\gamma_{ri2}(T_{i+1}-u)  = \gamma_{ri2}(T_{i} - u) + 2 =
   \gamma_{ri2}(T_i) + 2= \gamma_{ri2}(T_{i+1})$. If $u = x$, then clearly
   $\gamma_{ri2}(T_{i+1}-x)  = \gamma_{ri2}(T_{i} - x) + 2 =
   \gamma_{ri2}(T_i) + 2= \gamma_{ri2}(T_{i+1})$.
	If 	$u = v_1$, then $T_{i+1}$ is the union of $T_i$ and two isolated vertices and so $\gamma_{ri2}(T_{i+1}-v_1)  = \gamma_{ri2}(T_{i}) + 2 = \gamma_{ri2}(T_{i+1})$.
	
Assume now that $u=v_2$ (the case $u=v_3$ is similar). Clearly, any $\gamma_{ri2}$-function on $T_i$ can be extended by assigning a 1 to $v_1$ and a 2 to $v_3$ yielding $\gamma_{ri2}(T_{i+1}-v_2)\le \gamma_{ri2}(T_i) + 2$. Now let $f$ be a $\gamma_{ri2}(T_{i+1}-v_2)$-function. By Observation \ref{obv}, we may assume without loss of generality that $f(v_3)=1$. If $f(v_1)=0$, then we must have $f(x)=2$ and the function $f$ restricted to $T_i$ is a 2RiDF  on $T_i$ and we conclude from $x\in W_0({T_i})$ that $\gamma_{ri2}(T_{i+1}-v_2)=\omega(f)=\omega(f|_{T_i})+1\ge \gamma_{ri2}(T_i) + 2$. Assume that $f(v_1)=2$. Then we must have $f(x)=0$. If $2\in \{f(u)\mid u\in N_{T_i}(x)\}$, then the function $f$ restricted to $T_i$ is a 2RiDF  on $T_i$ and we have
$\gamma_{ri2}(T_{i+1}-v_2)=\omega(f)\ge \gamma_{ri2}(T_i) + 2$. Suppose $2\not\in \{f(u)\mid u\in N_{T_i}(x)\}$. Then the function $h:V(T_i)\to \{0,1,2\}$ defined by $h(x)=2$ and $h(u)=f(u)$ otherwise, is a 2RiDF  on $T_i$ and it follows from $x\in W_0({T_i})$ that $\gamma_{ri2}(T_{i+1}-v_2)=\omega(f)=\omega(h)+1\ge \gamma_{ri2}(T_i) + 2$. Thus $\gamma_{ri2}(T_{i+1}-v_2)\ge \gamma_{ri2}(T_i) + 2$. Hence $\gamma_{ri2}(T_{i+1}-v_2)=\gamma_{ri2}(T_i) + 2=\gamma_{ri2}(T_{i+1})$. Therefore $T_{i+1}$ is a $\gamma_{ri2}$-stable tree.
\end{proof}
\begin{lemma}\label{l2}
{\em If $T_i$ is  a $\gamma_{ri2}$-stable tree and a tree  $T_{i+1}$ is obtained from
$T_i$ by Operation ${\mathcal O}_2$, then
$T_{i+1}$ is a $\gamma_{ri2}$-stable tree.}
\end{lemma}
\begin{proof}
By Proposition \ref{spider}, $\gamma_{ri2}(T_{i+1}) = \gamma_{ri2}(T_{i}) +4$. For any vertex $u$ of $T_{i+1}$, we must show that $\gamma_{ri2}(T_{i+1}-u) = \gamma_{ri2}(T_{i+1})$. For any vertex $u \in V(T_i)-\{x\}$, by Proposition \ref{spider} we have
$\gamma_{ri2}(T_{i+1}-u) =\gamma_{ri2}(T_{i} - u) + 4 =
   \gamma_{ri2}(T_i) + 4= \gamma_{ri2}(T_{i+1})$.
	Also clearly $\gamma_{ri2}(T_{i+1}-x) =\gamma_{ri2}(T_{i} - x) + \gamma_{ri2}(P_7) =
   \gamma_{ri2}(T_i) + 4 =\gamma_{ri2}(T_{i+1})$.
For the vertex $v_1$ we have 	$\gamma_{ri2}(T_{i+1}-v_1)=\gamma_{ri2}(T_{i}) + 2\gamma_{ri2}(P_3) =
   \gamma_{ri2}(T_i) + 4 = \gamma_{ri2}(T_{i+1})$.
	Thus without loss of generality, it remains to show that
	$\gamma_{ri2}(T_{i+1}-u)  =  \gamma_{ri2}(T_{i+1})$ when $u \in \{v_2,v_3,v_4\}$.
Since $x\not\in W_0(T_i)$, there is a $\gamma_{ri2}$-function $g$ on $T_{i}$ with $g(x)>0$,
say without loss of generality that $g(x) = 1$.

First we show that $\gamma_{ri2}(T_{i+1} - v_2) = \gamma_{ri2}(T_{i+1})$. Clearly, the function $g$ can be extended to a 2RiDF  on $T_{i+1}-v_2$ by assigning a 2 to $v_3,v_5$, a 1 to $v_4,v_7$ and a 0 to $v_1,v_6$ and so $\gamma_{ri2}(T_{i+1} - v_2) \le \gamma_{ri2}(T_{i})+4= \gamma_{ri2}(T_{i+1})$.
Now let $h$ be a $\gamma_{ri2}(T_{i+1} - v_2)$-function. By Observation \ref{obv}, we have $v_3,v_4,v_7\in V_1\cup V_2$. Clearly, $h$ assigns a positive integer to at least one of $v_5$ or $v_6$. If $h(x)\neq 0$, then the function $h$ restricted to $T_i$ is a  2RiDF  on $T_{i}$ and hence $\gamma_{ri2}(T_{i+1} - v_2) \ge \gamma_{ri2}(T_{i})+4= \gamma_{ri2}(T_{i+1})$. If $h(x)=0$, then the function $h$ restricted to $T_i-x$ is a  2RiDF  on $T_{i}-x$ and hence $\gamma_{ri2}(T_{i+1} - v_2) \ge \gamma_{ri2}(T_{i}-x)+4\ge \gamma_{ri2}(T_{i})+4= \gamma_{ri2}(T_{i+1})$.

Now we show that $\gamma_{ri2}(T_{i+1} - v_3) = \gamma_{ri2}(T_{i+1})$. Clearly, the function $g$ can be extended to a 2RiDF  on $T_{i+1}-v_3$ by assigning a 2 to $v_5$, a 1 to $v_2,v_4,v_7$ and a 0 to $v_1,v_6$ and so $\gamma_{ri2}(T_{i+1} - v_3) \le \gamma_{ri2}(T_{i})+4= \gamma_{ri2}(T_{i+1})$. Now let $h$ be a $\gamma_{ri2}(T_{i+1} - v_2)$-function. Clearly $v_4\in V_1\cup V_2$ and by Observation \ref{obv} we have $v_2,v_7\in V_1\cup V_2$. Also, $h$ must assign a positive integer to at least one of $v_5$ or $v_6$. As above, we can see that
$\gamma_{ri2}(T_{i+1} - v_3) \ge \gamma_{ri2}(T_{i})+4= \gamma_{ri2}(T_{i+1})$.

Finally, we show that $\gamma_{ri2}(T_{i+1} - v_4) = \gamma_{ri2}(T_{i+1})$.
Clearly, the function $g$ can be extended to a 2RiDF  on $T_{i+1}-v_3$ by assigning a 2 to $v_3,v_5$, a 1 to $v_2,v_7$ and a 0 to $v_1,v_6$ and hence $\gamma_{ri2}(T_{i+1} - v_3) \le \gamma_{ri2}(T_{i})+4= \gamma_{ri2}(T_{i+1})$. Now let $h$ be a $\gamma_{ri2}(T_{i+1} - v_2)$-function. By Observation \ref{obv}, $v_3,v_7\in V_1\cup V_2$. Also, $h$ must assign a positive integer to at least one of $v_5$ or $v_6$ and a positive integer to at least one of $v_1$ or $v_2$. As above, we can see that
$\gamma_{ri2}(T_{i+1} - v_4) \ge \gamma_{ri2}(T_{i})+4= \gamma_{ri2}(T_{i+1})$. Therefore $T_{i+1}$ is a $\gamma_{ri2}$-stable tree.
\end{proof}
\begin{lemma}\label{l3}
{\em If $T_i$ is  a $\gamma_{ri2}$-stable tree and $T_{i+1}$ is a tree obtained from
$T_i$ by Operation ${\mathcal O}_3$, then
$T_{i+1}$ is a $\gamma_{ri2}$-stable tree.}
\end{lemma}
\begin{proof}
By Proposition \ref{spider}, we have $\gamma_{ri2}(T_{i+1}) = \gamma_{ri2}(T_{i}) +2k$. Now we show that $\gamma_{ri2}(T_{i+1}-u) = \gamma_{ri2}(T_{i+1})$ for each vertex $u$. If $u \in V(T_i)-\{x\}$, then by Proposition \ref{spider} we have
$\gamma_{ri2}(T_{i+1}-u) =\gamma_{ri2}(T_{i} - u) +2k =
   \gamma_{ri2}(T_i) + 2k= \gamma_{ri2}(T_{i+1})$.
For the vertex $x$ we have 	$\gamma_{ri2}(T_{i+1}-x)  = \gamma_{ri2}(T_{i} - x) + \gamma_{ri2}(S_k^{v_1}) =
\gamma_{ri2}(T_i) + 2k = \gamma_{ri2}(T_{i+1})$, and for the vertex $v_1$ we have
$\gamma_{ri2}(T_{i+1}-v_1)  = \gamma_{ri2}(T_{i}) + k\gamma_{ri2}(P_3) =
\gamma_{ri2}(T_i) + 2k = \gamma_{ri2}(T_{i+1})$.
Thus without loss of generality, it remains to prove that
$\gamma_{ri2}(T_{i+1}-u)  =  \gamma_{ri2}(T_{i+1})$ when $u \in \{v_1^1,v_2^1,v_3^1\}$.
	
Let $f$ be any $\gamma_{ri2}$-function  on $T_{i}$. Then $f$ can be extended to a 2RiDF  on $T_{i+1}-v_1^1$ by assigning a 2 to $v_2^1,v_1^2,v_3^3,\ldots,v_3^k$, a 1 to $v_3^1,v_3^2,v_1^3,\ldots,v_1^k$, and a 0 to $v_1,v_2^2,\ldots,v_2^k$, to a 2RiDF  on $T_{i+1}-v_2^1$ by assigning a 2 to $v_1^1,v_1^2,v_3^3,\ldots,v_3^k$, a 1 to $v_3^1,v_3^2,v_1^3,\ldots,v_1^k$, and a 0 to $v_1,v_2^2,\ldots,v_2^k$, and to a 2RiDF  on $T_{i+1}-v_3^1$ by assigning a 2 to $v_1^1,v_1^2,v_3^3,\ldots,v_3^k$, a 1 to $v_2^1,v_3^2,v_1^3,\ldots,v_1^k$, and a 0 to $v_1,v_2^2,\ldots,v_2^k$. This implies that $\gamma_{ri2}(T_{i+1}-u)\le \gamma_{ri2}(T_i) + 2k  =  \gamma_{ri2}(T_{i+1})$ when $u \in \{v_1^1,v_2^1,v_3^1\}$. Using an argument similar to that described in Lemma \ref{l2}, we can see that $\gamma_{ri2}(T_{i+1}-u) \ge  \gamma_{ri2}(T_{i+1})$ when $u \in \{v_1^1,v_2^1,v_3^1\}$.
Thus $\gamma_{ri2}(T_{i+1}-u)=\gamma_{ri2}(T_{i+1})$ when $u \in \{v_1^1,v_2^1,v_3^1\}$. Therefore $T_{i+1}$ is a $\gamma_{ri2}$-stable tree.
\end{proof}
\begin{theorem}\label{mainsuf}
Let T be a tree of order $n\geq 3$.  If $T\in \mathcal{T}$, then $T$ is a $\gamma_{ri2}$-stable tree.
 \end{theorem}

\begin{proof} Let $T\in \mathcal{T}$. We  proceed by induction on $l$, the number of operations used to construct $T$.
The base case is immediate by Observation \ref{esp}, since  either $T=P_3$
or $T=S^{v_1}_k\;(k\ge 3)$. Let $l \geq 1$ and suppose that
each tree $H$ in $\mathcal{T}$  which can be obtained from
a sequence of less than $l$ operations is a $\gamma_{ri2}$-stable tree.
Let $T \in \mathcal{T}$ and $T_1 \in \{P_3, S_k^{v_1}\}$,  $T_2, \ldots, T_{l+1}=T$  be a  sequence of trees
such that $T_{i+1}$ can be obtained from $T_i$ by one of the Operations
${\mathcal O}_1, {\mathcal O}_2$ or ${\mathcal O}_3$.
By the induction hypothesis, $T_{l}$ is a $\gamma_{ri2}$-stable tree.
Since $T=T_{l+1}$ is obtained  from $T_l$ by one of the operations
 ${\mathcal O}_1, {\mathcal O}_2$ or ${\mathcal O}_3$,
we conclude from  Lemmas \ref{l1}, \ref{l2} and \ref{l3} that
$T$ is a $\gamma_{ri2}$-stable tree.
\end{proof}

\begin{theorem}\label{mainnes}
Let T be a tree of order $n\geq 3$.  If $T$ is a 2-rainbow independent domination  stable tree,
then $T\in \mathcal{T}$.
 \end{theorem}
\begin{proof}
Let $T$ be a 2-rainbow independent domination  stable tree. The proof is by induction on $n$.
If $n=3$, then $T=P_3 \in \mathcal{T}$. Let $n\ge 4$ and let the result hold for all 2-rainbow independent domination stable trees $T$ of order less than $n$. Since $T$ is a 2-rainbow independent domination  stable tree, we deduce from Corollary \ref{n-1} that ${\rm diam}(T)\ge 3$.
If ${\rm diam}(T)=3$, then $T$ is a double star and we conclude from Proposition \ref{k14} (parts 1,6) that $T=DS_{2,2}$. Then $T$ can be obtained from $P_3$ by applying Operation $\mathcal{O}_1$ and so $T\in \mathcal{T}$.

Henceforth we assume that ${\rm diam}(T)\ge 4$. Let $v_1 \ldots v_k \;(k\ge 5)$ be a diametrical path in $T$ such that $\deg(v_2)$ is as large as possible and root $T$ at $v_k$. Since $T$ is a 2-rainbow independent domination  stable tree, we deduce from Proposition  \ref{k14} (item 6) that $\deg(v_2)\le 3$. We distinguish the following cases.

\smallskip
\noindent{\bf Case 1.} $\deg(v_2) = 3$.\\
Let $T'= T - T_{v_2}$ and $L_{v_2}=\{v_1,w\}$.
By Corollary \ref{leaf2} and by the assumption that $T$ is a 2-rainbow independent domination stable tree, for any $\gamma_{ri2}(T)$-function $f$ we have $f(v_2)=0$ and $f(v_1),f(v_2)\in \{1,2\}$.
Hence for any $v\in V(T')$, we deduce from Proposition \ref{k12} and by the assumption that
$\gamma_{ri2}(T')+2=\gamma_{ri2}(T)=\gamma_{ri2}(T-v)=\gamma_{ri2}(T'-v)+2$. This implies that $\gamma_{ri2}(T')=\gamma_{ri2}(T'-v)$  for each $v\in V(T')$. Hence $T'$ is a 2-rainbow independent domination stable tree and it follows from the induction hypothesis that $T'\in \mathcal{T}$.

We show next that $v_3\in W_0(T')$. If $v_3\not\in W_0(T')$ and $f_1$ is a $\gamma_{ri2}(T')$-function with $f_1(v_3)>0$, say $f_1(v_3)=1$, then $f_1$ can be extended to a 2IRDF of $T-w$ by assigning a 0 to $v_2$ and a 2 to $v_1$ implying that $\gamma_{ri2}(T- \{w\})\le \gamma_{ri2}(T)-1$, a contradiction. Thus $v_3\in W_0(T')$. Now $T$ can be obtained from $T'$ by Operation ${\mathcal O}_1$, yielding $T\in \mathcal{T}$.

\smallskip
\noindent{\bf Case 2.} $\deg(v_2) = 2$.\\
Since $T$ is a 2-rainbow independent domination stable tree, we deduce from Proposition \ref{k14} (parts 1,2) that  $\deg(v_3) = 2$. It follows from Proposition \ref{k14} (part 4) that  $\deg(v_4) \ge 3$.
We consider the following subcases.

\smallskip
{\bf Subcase 4.1.} $v_4$ has a children $y$ with depth 1.\\
Proposition \ref{k14} (parts 3,6) yields  $\deg(y) =3$. As in Case 1, we can see that $T\in \mathcal{T}$.

\smallskip
{\bf Subcase 4.2.} $v_4$ has a child $x$ with depth 0.\\
Since $T$ is a 2-rainbow independent domination stable tree, we deduce from Proposition \ref{k14} (part 7) that $\deg(v_4)\ge 4$. Let $f$ be a $\gamma_{ri2}(T)$-function. First let $v_4$ be a strong support vertex and let $y\in L(v_4)-\{x\}$. Then we have $f(v_1),f(x),f(y)\in \{1,2\}$. Since $T$ is a 2-rainbow independent domination stable tree, by Corollary \ref{leaf2} we must have $f(v_2)=f(v_4)=0$. Without loss of generality we may assume that $f(v_1)=f(x)=1$ and $f(v_3)=2$. Then the function $f$ restricted to $T-y$ is a 2RIDF of $T-y$ of weight less than $\gamma_{ri2}(T)$ which is a contradiction.
Now, suppose $v_4$ is not a strong support vertex. Considering above cases and subcases, we may assume that $T_{v_4}-x$ is an extended spider.
Then by Proposition \ref{k14} (part 5) we get a contradiction.
Assume that $v_4z_3z_2z_1$ is a path in $T$ such that $\deg(z_3)=\deg(z_2)=2$ and $\deg(z_1)=1$. Then we have $f(v_1),f(x),f(z_1)\in \{1,2\}$. Since $T$ is a 2-rainbow independent domination stable tree, we deduce from Corollary \ref{leaf2} that $f(v_2)=f(z_2)=f(v_4)=0$. Without loss of generality that we may assume $f(v_1)=f(x)=f(z_3)=1$ and $f(v_3)=f(z_1)=2$. Then the function $f$ restricted to $T-x$ is a 2RIDF of $T-x$ of weight less than $\gamma_{ri2}(T)$ which is a contradiction.

\smallskip
{\bf Subcase 4.3.} $T_{v_4}=S^{v_4}_{\deg(v_4)-1}$.\\
Let $T'=T-T_{v_4}$. Since $T$ is a 2-rainbow independent domination stable tree, we deduce from Proposition \ref{spider} that $T'$ is a 2-rainbow independent domination stable tree and by the induction hypothesis we have $T'\in \mathcal T$. If $\deg(v_4)\ge 4$, then $T$ can be obtained from $T'$ by Operation $\mathcal O_3$ and so $T\in \mathcal T$. Let $\deg(v_4)=3$ and $v_4z_3z_2z_1$ be a path in $T$ such that $\deg(z_3)=\deg(z_2)=2$ and $\deg(z_1)=1$. Now we show that $v_5\notin W_0(T')$. Since $T$ is a 2-rainbow independent domination stable tree, we have $\gamma_{ri2}(T-z_3)=\gamma_{ri2}(T)=\gamma_{ri2}(T')+4$. Let $f$ be a $\gamma_{ri2}(T-z_3)$-function. We may assume without loss of generality that $f(v_1)=f(z_1)$ and $f(z_2)=2$. Also we must have $f(v_2)\in \{1,2\}$ or $f(v_3)\in \{1,2\}$. If $f(v_4)=0$, to dominate $v_4$ we must have $f(v_5)\neq 0$ and then the function $f$ restricted to $T'$ is a $\gamma_{ri2}(T')$-function such that $f(v_5)\neq 0$ which implies that $v_5\notin W_0(T')$. Let $f(v_4)\neq 0$. If $f(v_5)\neq 0$ or $f(v_5)=0$ and $f(v_4)\in \{f(u)\mid u\in N_{T'}(v_5)\}$, then the function $g:V(T)\to \{0,1,2\}$ defined by $g(v_1)=g(z_3)=1$, $g(z_1)=g(v_3)=2$, $g(z_2)=g(v_4)=g(v_2)=0$ and $g(x)=f(x)$ otherwise, is a 2RIDF of $T$ of weight less than $\omega(f)$ which is a contradiction. Hence $f(v_5)=0$ and $f(v_4)\not\in \{f(u)\mid u\in N_{T'}(v_5)\}$. Then the function $h:V(T)\to \{0,1,2\}$ defined by $h(v_5)=f(v_4)$ and $h(x)=f(x)$ otherwise, is a $\gamma_{ri2}(T')$-function such that $h(v_5)\neq 0$ which implies that $v_5\notin W_0(T')$. Now $T$ can be obtained from $T'$ by Operation $\mathcal O_2$ and so $T\in \mathcal T$.
\end{proof}
Combining Theorems \ref{mainsuf} and \ref{mainnes}, we obtain Theorem \ref{main}.
\begin{corollary}
There exists  an $n$-order $\gamma_{ri2}$-stable tree if and only if
$n \in \{3,6,9,10,12,13\} \cup \{15,16,\ldots\}$.
\end{corollary}


\section{Edge removal: trees}
In this section we study the effect of edge removal on 2-rainbow independent domination number in trees. We begin with a simple proposition.

\begin{proposition}\label{edgedel}
Let $e=xy$  be an edge of a tree $T$. Then
\begin{itemize}
\item[(i)]  $\gamma_{ri2}(T) \leq  \gamma_{ri2}(T-e) \leq \gamma_{ri2}(T) + 1$.
\item[(ii)]  $\gamma_{ri2}(T-e) = \gamma_{ri2}(T) + 1$ if and only if for each
                    $\gamma_{ri2}$-function $f=(V_0; V_1; V_2)$ on $T$ exactly one of
										$x$ and $y$ is in $V_0$ and this one has exactly $2$ neighbors in $V_1 \cup V_2$.
\end{itemize}
\end{proposition}
\begin{proof}  (i) Let $f = (V_0; V_1; V_2)$ be an arbitrary  $\gamma_{ri2}$-function on $T-e$.
 Clearly, we can choose $f$ so that $f(x) \not= f(y)$ when $x,y \in V_1 \cup V_2$.
But then $f$ is a 2RiDF on $T$, which implies $\gamma_{ri2}(T) \leq \gamma_{ri2}(T-e)$.

To prove the right inequality, let $f$ be a $\gamma_{ri2}$-function on $T$. If $0 \not\in \{f(x),f(y)\}$ or $f(x)=f(y)=0$, then $f$ is a 2RiDF on $T-e$ and so $\gamma_{ri2}(T-e) \leq \gamma_{ri2}(T)$. Assume without loss of generality that $f(x)=1$ and $f(y)=0$. If $1\in \{f(u)\mid u\in N(y)-\{x\}\}$, then $f$ is a 2RiDF on $T-e$ and so $\gamma_{ri2}(T-e) \leq \gamma_{ri2}(T)$. If $1\not\in \{f(u)\mid u\in N(y)-\{x\}\}$, then the function $g:V(T-e)\to \{0,1,2\}$ defined by $g(y)=1$ and $g(u)=f(u)$ otherwise, is a $\gamma_{ri2}$-function on $T-e$ and so $\gamma_{ri2}(T-e) \leq \gamma_{ri2}(T)+1$. Thus $\gamma_{ri2}(T-e) \leq \gamma_{ri2}(T)+1$.

(ii) Let $\gamma_{ri2}(T-e) = \gamma_{ri2}(T) + 1$ and let $f = (V_0; V_1; V_2)$ be an arbitrary $\gamma_{ri2}$-function on $T$. A close look at the proof of $(i)$, shows that exactly one of $x$ and $y$ is in $V_0$.
			Suppose without loss of generality that $x \in V_1$ and $y \in V_0$. Then we must have
			$N(y) \cap V_1 = \{x\}$. Root $T$ at $x$. If $|N(y) \cap V_2|\ge 2$ and $z_1,z_2 \in N(y) \cap V_2$, then the function $g:V(T-e)\to \{0,1,2\}$ defined by $g(u)=1$ if $u\in T_{z_1}$ and $f(u)=2$, $g(u)=2$ if $u\in T_{z_1}$ and $f(u)=1$, and $g(u)=f(u)$ otherwise, is a $\gamma_{ri2}$-function on $T-e$ yielding
$\gamma_{ri2}(T-e) \leq \gamma_{ri2}(T)$ which is a contradiction. Thus $y$ has exactly $2$ neighbors in $V_1 \cup V_2$.
				
Conversely, let for any $\gamma_{ri2}$-function $f=(V_0; V_1; V_2)$ on $T$ exactly one of $x$ and $y$ is in $V_0$ and this one has exactly $2$ neighbors in $V_1 \cup V_2$. Suppose, to the contrary, that $\gamma_{ri2}(T-xy) \not= \gamma_{ri2}(T) + 1$.
       By (i) we have  $\gamma_{ri2}(T-xy) = \gamma_{ri2}(T)$. Let $l = (V_0^l; V_1^l; V_2^l)$ be a
			$\gamma_{ri2}$-function on $T-xy$. Clearly, we can choose $l$ so that
			$l(x) \not= l(y)$ when neither $l(x)$ nor $l(y)$ is $0$. Then $l$ is a 2RiDF on
			$T$ and since $\gamma_{ri2}(T-xy) = \gamma_{ri2}(T)$,  $l$ is a  $\gamma_{ri2}$-function on $T$. But, $l$ does not satisfy in the assumption which is a contradiction.
\end{proof}
A tree $T$ is called a $\gamma_{ri2}$-\emph{ER-critical tree} if for each edge $e$ of $T$, $\gamma_{ri2}(T-e) = \gamma_{ri2}(T) + 1$.
In what follows we give necessary and sufficient conditions for a tree
to be  $\gamma_{ri2}$-ER-critical.
Let $\mathcal F=\{S(T)\mid T\;{\rm is\;a\;non-trivial\;tree}\}$. Clearly, $F$ is a family of trees.

\begin{theorem}\label{suf}
If $T\in \mathcal F$, then $T$ is a $\gamma_{ri2}$-ER-critical tree.
\end{theorem}
\begin{proof}
By definition, there exists a non-trivial tree $T'$ such that $T=S(T')$. Let $X$ be the set of subdivision vertices and $Y=V(T')$. Clearly, all leaves are in $Y$. Let $u$ be a leaf. Define a function $f : V(T ) \rightarrow \{0, 1, 2\}$ by $f(v)  = 1$ when $d(u,v)\equiv 0\pmod 4$, $f(v) = 2$ when  $d(u,v)\equiv 2\pmod 4$ and $f(v) = 0$ for $v \in X$. Since  there is a unique path between any $2$ vertices of a tree and
since the distance between any $2$ leaves is even, the function $f$ is well defined and clearly it is a 2RiDF on $T$ with $V_1^f \cup V_2^f = Y$.  	 Hence $|Y| \geq \gamma_{ri2}(T)$.

Assume that $h = (V_0; V_1;V_2)$ be an arbitrary $\gamma_{ri2}$-function  on $T$. We claim that $V_1\cup V_2=Y$. Suppose, to the contrary, that
$A=X\cap (V_1\cup V_2)$ and  $B=Y\cap (V_1\cup V_2)$.  	
Since each vertex in $A$ has degree $2$, the number of edges between $A$ and $Y-B$
is at most $2|A|$. On the other hand, since each vertex in $Y-B$ is adjacent to
at least two vertices in $A$, the number of edges between $A$ and $Y-B$
is at least $2|Y-B|$. It follows that $|A| \geq |Y|-|B|$. Thus
$$\gamma_{ri2}(T) = |V_1 \cup V_2| = |A \cup B|= |A| + |B| \geq |Y|  \geq 	\gamma_{ri2}(T).$$
Hence all inequalities occurring in above chain become equalities and so $|Y|=\gamma_{ri2}(T)$ and $|A| =|Y|-|B| = |Y-B|$.
This implies that $N(A) = Y-B$. Since $T$ is connected, there are two vertices $y \in Y-B$ and $z \in X-A$ such that $yz\in E(T)$.
But then $z$ is adjacent to exactly one vertex in $B$, and so $z$ is not dominated by $h$. Thus $V_1\cup V_2=Y$ and $V_0=X$. We conclude from  Proposition \ref{edgedel} that $T$ is a $\gamma_{ri2}$-ER-critical.
\end{proof}

The next result is an immediate consequence of the proof of Theorem \ref{suf}.
\begin{proposition}\label{dom2}
If $T$ is a non-trivial tree then $\gamma_{ri2}(S(T)) = |V(T)|$.
\end{proposition}

\begin{theorem}\label{nes}
Let $T$ be a non-trivial tree. Then $T\in \mathcal F$ if and only if $T$ is a $\gamma_{ri2}$-ER-critical tree.
\end{theorem}
\begin{proof}
According to Theorem \ref{suf}, we need to prove necessity. Let $T$ be a $\gamma_{ri2}$-ER-critical tree and $f$ a $\gamma_{ri2}$-function. Suppose $x^1_1$ is a leaf of $T$, $y^1_1$ is its neighbor and root $T$ at $x^1_1$. We may assume without loss of generality that $f(x^1_1)=1$. Since $T$ is $\gamma_{ri2}$-ER-critical, we deduce from Proposition \ref{edgedel} that $\deg(y^1_1)=2$ and $f(y^1_1)=0$. Assume that $N(y^1_1)=\{x^1_1,x^2_1\}$. To dominate $y^1_1$, we must have $f(x^2_1)=2$. Let $N(x^2_1)-\{y^1_1\}=\{y^2_1,\ldots,y^2_k\}$ if $N(x^2_1)-\{y^1_1\}\neq\emptyset$. Since $T$ is $\gamma_{ri2}$-ER-critical, we conclude from Proposition \ref{edgedel} that $\deg(y^2_1)=\cdots=\deg(y^2_k)=2$ and $f(y^2_1)=\cdots=f(y^2_k)=0$. Assume 
$N(y^2_i)=\{x_1^2,x^3_i\}$
for each $i$. Then we must have $f(x^3_1)=\cdots=f(x^3_k)=1$. By continuing this process, we have (i) each vertex with odd distance from $x_1^1$ is of degree 2, labeled with $y^i_j$ for some $i,j$, and is assigned 0 under $f$, (ii) each vertex with distance $4s+2\;(s\ge 0)$ from $x_1^1$ is labeled with $x^i_j$ for some $i,j$, and is assigned 2 under $f$, and (ii) each vertex with distance $4s\;(s\ge 1)$ from $x_1^1$ is labeled with $x^i_j$ for some $i,j$, and is assigned 1 under $f$. Let $T'$ be a tree obtained from $T$ by removing any vertex with label $y^i_j$ and joining the vertices adjacent to $y^i_j$. Then $T=S(T')$ and so $T\in \mathcal F$. This completes the proof.
\end{proof}

\begin{corollary}\label{ccc}
There exists  an $n$-order $\gamma_{ri2}$-ER-critical tree $T$ if and only if $n \geq 3$ and $n$  is odd.
Moreover, $\gamma_{ri2}(T) = \left\lceil \frac{n+1}{2} \right\rceil$.
\end{corollary}

\begin{proof}
Immediately by Theorem \ref{nes}  and Proposition \ref{dom2}.
\end{proof}
S. Brezovnik and  T.K. \v{S}umenjak \cite{BS} presented the following bounds on 2-rainbow independent domination numbers of trees.
\begin{theorem}\label{BS}
{\em For any tree $T$ of order $n$ with $\ell$ leaves, $$\frac{n+1}{2}\le \gamma_{rik}(T)\le \frac{n+\ell}{2}.$$}
\end{theorem}
By Corollary \ref{ccc} and Theorem \ref{BS}, we see that $\gamma_{ri2}$-ER-critical
trees attains minimum value of 2-rainbow independent domination number among all trees.


\section{Open problems and questions}
We conclude the paper by some  problems and directions for further research.
\begin{enumerate}
\item Find sharp lower and upper bounds for 2-rainbow independent domination number of a connected
                                $\gamma_{ri2}$-stable ($\gamma_{ri2}$-ER-critical) graph in terms of its order.
  \item What is the minimum/maximum number of edges of a  connected
                                $\gamma_{ri2}$-stable ($\gamma_{ri2}$-ER-critical) graph
                                with  a given order and  a 2-rainbow independent domination number?

  \item Characterize  all unicyclic  $\gamma_{ri2}$-stable ($\gamma_{ri2}$-ER-critical) graphs.
\end{enumerate}

\end{document}